\documentclass[reqno,11pt]{amsart}
 
\usepackage{amsmath}
\usepackage{amsthm}  
\usepackage{verbatim}
\usepackage{enumerate} 
\usepackage{mathtools}  
\usepackage{amssymb}
\usepackage{mathrsfs}
\usepackage[top=1.5in, bottom=1.5in, left=0.7in, right=0.7in]{geometry} 
\usepackage[colorlinks]{hyperref}
\hypersetup{
	colorlinks,
	citecolor=blue,
	linkcolor=red
}   
\numberwithin{equation}{section}
\usepackage{xypic}
\usepackage{bm}
\usepackage{tikz}
\usetikzlibrary{decorations.pathreplacing}
\usepackage{xcolor}
\usepackage{float}
\usepackage{cleveref}
\usepackage{comment}

\newtheorem{theorem}{\textbf{Theorem}}[section]
\newtheorem{theorem*}{\textbf{Theorem}}

\newtheorem{lemma}[theorem]{\textbf{Lemma}}
\newtheorem{question}[theorem]{\textbf{Question}}

\newtheorem{definition/proposition}[theorem]{\textbf{Definition/Proposition}}

\def\R{\mathbb{R}}
\def\Z{{\mathbb Z}}

\def\rd{{\rm d}}

\def\la{\langle\,}
\def\ra{\,\rangle}

\DeclareMathOperator{\Id}{id}

\DeclareMathOperator{\coker}{coker}

\newcommand{\Addresses}{{% additional braces for segregating \footnotesize
		\bigskip
		\footnotesize

	     Zhengyi Zhou, \par\nopagebreak
	    \textsc{Morningside Center of Mathematics and Institute of Mathematics, AMSS, CAS, China}\par\nopagebreak
		\textit{E-mail address}: \href{mailto:zhyzhou@amss.ac.cn}{zhyzhou@amss.ac.cn}

}}

\title{A note on contact manifolds with infinite fillings}
\author{Zhengyi Zhou}

\begin{document}
	\maketitle
\begin{abstract}
We use spinal open books to construct contact manifolds with infinitely many different Weinstein fillings in any odd dimension $> 1$, which were previously unknown for dimensions equal to $4n+1$. The argument does not involve understanding factorizations in the symplectic mapping class group.
\end{abstract}
%\tableofcontents
\section{Introduction}
Contact manifolds arise naturally as convex boundaries of symplectic manifolds, it was known by Gromov \cite{zbMATH03951473} and Eliashberg \cite{zbMATH04209004} in the late 1980s that not all contact manifolds can be realized in such a way. Therefore understanding symplectic fillings of contact manifolds is a fundamental question in contact topology. Such questions were extensively studied by many researchers starting from  the case of no fillings \cite{zbMATH06111802,zbMATH06385143,zbMATH07601288,zbMATH00892252,zbMATH01777256,zbMATH02206566,zbMATH06182635,zbMATH07442197}, the case of unique (only for the topological type in many cases) fillings \cite{zbMATH07131053, zbMATH07601288,geiges2019diffeomorphism,zbMATH03951473,zbMATH04186536,zbMATH05688243,MR4566698}, and at the end of this spectrum, the case of infinitely many filings \cite{zbMATH06862804,zbMATH02041154}. As one can always blow up a symplectic filling to change the topology, we need to restrict to Liouville or Weinstein fillings for the question of infinite fillings. Unlike the no-filling and unique-filling situations, which typically depend on some rigidity arguments using pseudo-holomorphic curves,  the construction of contact manifolds of infinitely many fillings usually uses the topological or flexible side of symplectic topology. The first contact manifold (beyond the trivial case of $S^1$) with infinitely many different Weinstein fillings was constructed by Ozbagci and Stipsicz \cite{zbMATH02041154} in dimension 3. Nowadays, there are many constructions with various constraints on the topology of fillings in dimension $3$, see \cite{zbMATH06440197,zbMATH05505943,zbMATH06924053,zbMATH06537657}. Oba \cite{zbMATH06862804} generalized Ozbagci and Stipsicz's result to dimension $4n-1$. Their constructions were based on the open book construction of contact manifolds and  finding infinitely many different factorizations by positive Dehn-Seidel twists of the monodromy in the symplectic mapping class group. Such an approach is most efficient in dimension $3$, as the symplectic mapping class group agrees with the classical mapping class group in the case of surfaces.  In higher dimensions, the symplectic mapping class group is different from the classical mapping class group in general, and much less is known. It is worth pointing out that Lazarev \cite{zbMATH07197504} constructed contact manifolds with many different Weinstein fillings in dimension $\ge 5$, where the number of fillings can be arbitrarily large, from a surgical perspective.

In this note, we give a new construction of contact manifolds with infinitely many different Weinstein fillings in any dimension. The construction is based on  spinal open books--a generalization of contact open books, introduced by Lisi, Van Horn-Morris and Wendl \cite{lisi2018symplectic}. The spinal open book was used by Baykur and Van Horn-Morris \cite{zbMATH06537657} to construct contact 3-manifolds which admit infinitely many Weinstein fillings with arbitrarily big Euler characteristics and arbitrarily small signatures. Heuristically speaking, spinal open books arise as the contact boundary of a Lefschetz fibration over a general surface with boundary. Such contact manifolds, especially in dimension $4$,  were studied systematically in \cite{lisi2018symplectic,lisi2020symplectic,MR4278702}. Moreover, there are notions of spinal open books, which fiber over Liouville domains of dimension higher than $2$, see e.g.\ \cite{MR4045357,moreno2017algebraic}. In this note, we restrict to the case of the surface base. 
\begin{theorem}\label{thm:main}
Let $V$ be the plumbing of two $T^*S^n$ along three points. Then the contact boundary $\partial(\Sigma_{1.1} \times V)$ has infinitely many different Weinstein fillings, where $\Sigma_{1,1}$ is a genus one Riemann surface with one boundary component, viewed as a Weinstein filling of $S^1$.
\end{theorem}
We point out that our construction is local in nature, i.e.\ if $V$ contains the domain in \Cref{thm:main} as a symplectic subdomain and the homology computation in the proof works, then the conclusion can be drawn for  $\partial(\Sigma_{1.1} \times V)$. For example, it holds for $V$ in \Cref{thm:main} taking boundary connected sum with any Weinstein domain. Moreover, similar phenomena hold for more general plumbings of spheres. Moreover, when $n$ is odd, i.e.\ when the contact manifold is of dimension $4k-1$, we can take $V$ to be two $T^*S^n$ plumbed at one point.

Our strategy is similar to \cite{zbMATH06862804,zbMATH02041154}: the contact boundary $\partial(\Sigma_{1.1} \times V)$ is the trivial spinal open book over $\Sigma_{1,1}$, and any representation $\rho:\pi_1(\Sigma_{1,1})\to Symp_c(V)$ such that $\rho(\partial \Sigma_{1,1}) = \Id$ gives rise to a Weinstein filling by a $V$-fiber bundle over $\Sigma_{1,1}$. Sending a generator of $\pi_1(\Sigma_{1,1})$ to identity always yields one such representation, hence every element $\phi$ of $Symp_c(V)$, i.e.\ the image of the other generator under $\rho$, induces a filling. Then by understanding the effect of $\phi$ on the  homology of $V$, we can get infinite many fillings. In particular, we do not need to consider factorizations in the symplectic mapping class group.

From \Cref{thm:main}, the next natural question is depriving the question of the classical topological aspect, namely:
\begin{question}\label{question}
	Are there contact manifolds with infinitely many different Weinstein/Liouville fillings with the same formal data, i.e.\  as the same differential/almost complex/almost Weinstein manifold (relative to the boundary)?
\end{question}	
We expect the answer to the question to be yes, at least for dimensions high enough. However, to the best of the author's knowledge, we do not even know examples of contact manifolds with more than one smoothly same, but symplectically different fillings. Unlike \Cref{thm:main}, rigidity techniques, e.g.\ holomorphic curves or sheaves, must enter the picture to solve the above question. Using spinal open books, we have candidates for  Question \ref{question} at least for in dimension $4n+1$.

\begin{question}
Let $\phi\in Symp_c(V)$  be generated by eighth powers of the Dehn-Seidel twist along Lagrangian spheres in $V$, where $\dim V = 4n$. Is the symplectic fiber bundle induced from $\pi_1(\Sigma_{1,1})\to Symp_c(V)$ by sending one generator to $\phi$ and the other to $\Id$ symplectomorphic to $\Sigma_{1,1}\times V$?
\end{question}	

Clearly, the motivation behind such a question is the fact that eighth powers of the Dehn-Seidel twist are smoothly isotopic to identity in dimension $4n$ \cite{klein2011symplectic}, yet not symplectically isotopic to identity \cite{zbMATH07131053,zbMATH01465075}. In dimensions $4$ and $12$, one can replace the eighth power with a square.

\subsection*{Acknowledgments}
The author is grateful to Fabio Gironella for productive discussions and interest in the project, and to Samuel Lisi, Takahiro Oba, Chris Wendl for helpful comments. The author is supported by National Natural Science Foundation of China under Grant No.\ 12288201 and 12231010.

\section{Proof}
Let $V$ be a Liouville domain and $\phi\in \pi_0(Symp_c(V))$. We can endow 
$$\Sigma_{g,1}\times \partial V \cup_{S^1\times \partial V} V_{\phi}$$
a contact structure by the Thurston-Winkelnkemper construction, see \cite[\S 2.1]{lisi2018symplectic}, where $\Sigma_{g,1}$ is a genus $g$ surface with one boundary component and $V_{\phi}$ is the mapping torus $V\times [0,1]/ (x,0)\sim (\phi(x),1)$. This is a very special case of the spinal open book considered in \cite{lisi2018symplectic}, where the vertebrae ($\Sigma_{g,1}$ here) can have more boundary components and be disconnected. In this paper, we only consider the case of $g=1$ and $\phi=\Id$. Then the contact manifold is the contact boundary $\partial (\Sigma_{1,1}\times V)$. 

\begin{lemma}[{\cite{zbMATH06537657,lisi2018symplectic}}]\label{lemma:filling}
Let $\Sigma$ be a connected Riemann surface with boundary and $V$ be a Weinstein domain, any representation $\pi_1(\Sigma) \to \pi_0(Symp_c(V))$ mapping the boundary to $\Id$ gives rise to a Weinstein filling of $\partial (\Sigma \times V)$, which is diffeomorphic to the $V$-bundle over $\Sigma$ from  $\pi_1(\Sigma) \to \pi_0(Symp_c(V))$. 
\end{lemma}	

More generally, if the monodromy of the spinal open book is $\phi$ and there exist $\phi_1,\psi_1,\ldots, \phi_{g}, \psi_{g}\in Symp_c(V)$ and $\tau_1,\ldots,\tau_k$ are Dehn-Seidel twists along some exact Lagrangian spheres in $V$, such that 
$$\phi= \prod \tau_i \prod [\phi_{i},\psi_{i}]$$
Then the spinal open book given by $(\Sigma_{g,1},V,\phi)$ is the contact boundary of a symplectic Lefschetz fibration over $\Sigma_{g,1}$ with $k$ singular fibers. When $V$ is Weinstein, the total space of the Lefschetz fibration is a Weinstein filling of the spinal open book. 

\begin{lemma}\label{lemma:homology}
Let $V_{\phi}$ be the mapping torus, then we have short exact sequences $$0\to \ker ( \phi_*-\Id) \to H_*(V_{\phi}) \to \coker (\phi_*-\Id)[-1] \to 0$$
\end{lemma}
\begin{proof}
	The homology of $V_{\phi}$ can be computed from the cone of $C_*(V)\stackrel{\phi_*-\Id}{\longrightarrow} C_*(V)$. The induced long exact sequence implies the short exact sequences above.
\end{proof}	
More generally, let $V_{\phi\vee \Id}$ be the $V$-fiber bundle over $S^1\vee S^1$ (or homotopically equivalently over $\Sigma_{1,1}$), such that the monodromy over one $S^1$ is $\phi$ and is $\Id$ over the other $S^1$. Then we have a short exact sequence 
$$0\to \ker(\phi_*-\Id)|_{H_k(V;\Z)} \to  H_k(V_{\phi \vee \Id};\Z) \to H_{k-1}(V;\Z) \oplus \coker( \phi_*-\Id)|_{H_{k-1}(V;\Z)}\to 0$$
for $k\ge 1$. In particular, when $V$ is a Weinstein domain of dimension $2n$, then the cardinality of the torsion of $H_{n+1}(V_{\phi\vee\Id})$ will be at least that of $\coker( \phi_*-\Id)|_{H_{n}(V;\Z)}$

\begin{lemma}[{Picard-Lefschetz formula, \cite[(6.3.3)]{zbMATH03695355}}]\label{lemma:PL}
Let $L$ be a Lagrangian $n$-sphere in an exact domain $W$ and $\tau_L$ the Dehn-Seidel twist along $L$, then $(\tau_L)_*:H_*(W;\Z)\to H_*(W;\Z)$ is given by
$$(\tau_L)_*(c) = \left\{ 
\begin{array}{ll}
     c+(-1)^{\frac{(n+1)(n+2)}{2}}\la c,[L] \ra [L], & c\in H_n(W;\Z); \\
     c, &c\in H_j(W;\Z), j\ne n. 
\end{array}\right.$$
where $\la \cdot, \cdot \ra:H_n(W;\Z)\otimes H_n(W;\Z)\to \Z$ is the the intersection product.
\end{lemma}

Let $V^{2n}$ be plumbing of two $T^*S^n$ along three points. We use $L_1, L_2$ to denote the two Lagrangian spheres, oriented such that $\la [L_1],[L_2]  \ra = (-1)^{\frac{n(n+1)}{2}}3$\footnote{That is we orient $L_2$ by the induced orientation of $T^*_qL_1$, where $q$ is an intersection point. Then the intersection number using the orientation on $T^*L_1$ induced from the orientation of $L_1$ is $3$. The extra sign comes from that the symplectic orientation (using $-\rd \sum p_i\rd q_i$, i.e.\ the standard symplectic orientation on $\R^{2n}=T^*\R^n$) is different from the induced orientation on $T^*L_1$ from that of $L_1$ by  $(-1)^{\frac{n(n+1)}{2}}$.}. When $n>1$, under the free basis $[L_1], [L_2]$ of $H_{n}(V^{2n};\Z)$, by the Picard-Lefschetz formula, the effect of the Dehn-Seidel twists $\tau_{L_1},\tau_{L_2}$ on $H_n(V^{2n};\Z)$ is given by
    $$\left[
    \begin{array}{cc}
        1 &  -3 \\
        0 & 1
    \end{array}\right],
    \left[
    \begin{array}{cc}
        1 & 0\\
        3 & 1
    \end{array}\right]
    $$
    for $n$ odd respectively, and 
    $$    \left[
        \begin{array}{cc}
        -1 & -3 \\
        0 & 1
        \end{array}\right]
    \left[
    \begin{array}{cc}
    1 & 0 \\
    -3 & -1
    \end{array}\right]
    $$
    for $n$ even respectively. When $n=1$, $H_{1}(V;\Z)=\Z^4$ and $(\tau_{L_1})_*$ using the basis $[L_1],[L_2]$ and two other cycles (with suitable orientation) glued from two arcs from $L_1,L_2$ is given by
    $$   \left[
    \begin{array}{cccc}
    1 & -3 & -1 & -1\\
    0 & 1  & 0 & 0\\
    0 & 0 & 1 & 0 \\
    0 & 0 & 0 & 1
    \end{array}\right]
    $$
    
\begin{proof}[Proof of \Cref{thm:main}]
Let $\gamma_1,\gamma_2$ be two loops in $\Sigma_{1,1}$, representing the bases of the fundamental group in the torus. We consider the representation $\rho:\pi_1(\Sigma_{1,1})\to Symp_c(V), \gamma_1\mapsto \phi, \gamma_2\mapsto \Id$. By \Cref{lemma:filling}, it gives rise to a filling of $\partial(\Sigma_{1,1}\times V)$, which is homotopy equivalent to $V_{\phi\vee \Id}$. 

When $n>1$ is odd, we take $\phi$ to be $\tau_{L_1}$. Since $\phi_*^k$ on $H_n(V;\Z)$ is given by 
	 $$\left[\begin{array}{cc}
		1 & -3k \\
		0 & 1
	\end{array}\right] 
	$$
Then by the discussion after \Cref{lemma:homology}, we know that $H_{n+1}(V_{\phi^k\vee \Id};\Z)$ has a torsion of $\Z/3k$. As a consequence, each $k$ yields a different Weinstein filling.

When $n=1$, we take $\phi$ to be $\tau_{L_1}$.  Then $\phi_*^k$ on $H_n(V;\Z)$ is given by 
    $$   \left[
    \begin{array}{cccc}
    1 & -3k & -k & -k\\
    0 & 1  & 0 & 0\\
    0 & 0 & 1 & 0 \\
    0 & 0 & 0 & 1
    \end{array}\right]
    $$
We know that $H_{2}(V_{\phi^k\vee \Id};\Z)$ has a torsion of $\Z/k$.  As a consequence, each $k$ yields a different Weinstein filling.

When $n$ is even, we take $\phi$ to be $\tau_{L_1}\circ \tau_{L_2}$. Then $\phi_*$ on $H_n(V;\Z)$ is given by 
$$\left[\begin{array}{cc}
	8 & 3 \\
	-3 & -1
\end{array}\right] $$
This matrix has positive  eigenvalues $\lambda_1 = \frac{7+3\sqrt{5}}{2}>1, \lambda_2=\frac{7-3\sqrt{5}}{2}<1$. As a consequence, we have 
	$$|\det((\phi_*)^k-\Id)|=|2-\lambda_1^k-\lambda_2^k|,$$
	which grows exponentially. The the torsion of $H_{n+1}(V_{\phi^k\vee \Id})$ is of size $|2-\lambda_1^k-\lambda_2^k|$, which yields infinitely many different fillings as before.
\end{proof}
When $n$ is odd, we can simply take $V$ to be the plumbing of $T^*S^n$ at one point. Then $\tau^k_{L_1}$ acts on $H_n(V;\Z)$ by 
	 $$\left[\begin{array}{cc}
		1 & -k \\
		0 & 1
	\end{array}\right] 
	$$
which yields infinitely many fillings.
\bibliographystyle{plain} 
\bibliography{ref}

\begin{thebibliography}{10}

\bibitem{zbMATH06440197}
Selman Akbulut and Kouichi Yasui.
\newblock Infinitely many small exotic {Stein} fillings.
\newblock {\em J. Symplectic Geom.}, 12(4):673--684, 2014.

\bibitem{zbMATH05505943}
Anar Akhmedov, John~B. Etnyre, Thomas~E. Mark, and Ivan Smith.
\newblock A note on {Stein} fillings of contact manifolds.
\newblock {\em Math. Res. Lett.}, 15(5-6):1127--1132, 2008.

\bibitem{zbMATH06924053}
Anar Akhmedov and Burak Ozbagci.
\newblock Exotic {Stein} fillings with arbitrary fundamental group.
\newblock {\em Geom. Dedicata}, 195:265--281, 2018.

\bibitem{zbMATH07131053}
Kilian Barth, Hansj{\"o}rg Geiges, and Kai Zehmisch.
\newblock The diffeomorphism type of symplectic fillings.
\newblock {\em J. Symplectic Geom.}, 17(4):929--971, 2019.

\bibitem{zbMATH06537657}
R.~Inanc Baykur, Jeremy van Horn-Morris, Samuel Lisi, and Chris Wendl.
\newblock Families of contact 3-manifolds with arbitrarily large {Stein}
  fillings.
\newblock {\em J. Differ. Geom.}, 101(3):423--465, 2015.

\bibitem{zbMATH06111802}
Jonathan Bowden.
\newblock Exactly fillable contact structures without {Stein} fillings.
\newblock {\em Algebr. Geom. Topol.}, 12(3):1803--1810, 2012.

\bibitem{zbMATH06385143}
Jonathan Bowden, Diarmuid Crowley, and Andr{\'a}s~I. Stipsicz.
\newblock The topology of {Stein} fillable manifolds in high dimensions. {I}.
\newblock {\em Proc. Lond. Math. Soc. (3)}, 109(6):1363--1401, 2014.

\bibitem{zbMATH07601288}
Jonathan Bowden, Fabio Gironella, and Agustin Moreno.
\newblock Bourgeois contact structures: tightness, fillability and
  applications.
\newblock {\em Invent. Math.}, 230(2):713--765, 2022.

\bibitem{zbMATH04209004}
Yakov Eliashberg.
\newblock Filling by holomorphic discs and its applications.
\newblock Geometry of low-dimensional manifolds. 2: {Symplectic} manifolds and
  {Jones}-{Witten}-{Theory}, {Proc}. {Symp}., {Durham}/{UK} 1989, {Lond}.
  {Math}. {Soc}. {Lect}. {Note} {Ser}. 151, 45-72 (1990)., 1990.

\bibitem{zbMATH00892252}
Yasha Eliashberg.
\newblock Unique holomorphically fillable contact structure on the 3-torus.
\newblock {\em Int. Math. Res. Not.}, 1996(2):77--82, 1996.

\bibitem{zbMATH01777256}
John~B. Etnyre and Ko~Honda.
\newblock Tight contact structures with no symplectic fillings.
\newblock {\em Invent. Math.}, 148(3):609--626, 2002.

\bibitem{geiges2019diffeomorphism}
Hansj{\"o}rg Geiges, Myeonggi Kwon, and Kai Zehmisch.
\newblock Diffeomorphism type of symplectic fillings of unit cotangent bundles.
\newblock {\em arXiv preprint arXiv:1909.13586}, 2019.

\bibitem{zbMATH02206566}
Paolo Ghiggini.
\newblock Strongly fillable contact 3-manifolds without {Stein} fillings.
\newblock {\em Geom. Topol.}, 9:1677--1687, 2005.

\bibitem{zbMATH03951473}
M.~Gromov.
\newblock Pseudo holomorphic curves in symplectic manifolds.
\newblock {\em Invent. Math.}, 82:307--347, 1985.

\bibitem{klein2011symplectic}
Andreas Klein.
\newblock Symplectic monodromy, leray residues and quasi-homogeneous
  polynomials.
\newblock {\em arXiv preprint arXiv:1101.3554}, 2011.

\bibitem{zbMATH03695355}
Klaus Lamotke.
\newblock The topology of complex projective varieties after {S}. {Lefschetz}.
\newblock {\em Topology}, 20:15--51, 1981.

\bibitem{zbMATH07197504}
Oleg Lazarev.
\newblock Maximal contact and symplectic structures.
\newblock {\em J. Topol.}, 13(3):1058--1083, 2020.

\bibitem{lisi2018symplectic}
Samuel Lisi, Van Horn-Morris, and Chris Wendl.
\newblock On symplectic fillings of spinal open book decompositions {I}:
  Geometric constructions.
\newblock {\em arXiv preprint arXiv:1810.12017}, 2018.

\bibitem{lisi2020symplectic}
Samuel Lisi, Van Horn-Morris, and Chris Wendl.
\newblock On symplectic fillings of spinal open book decompositions i{I}:
  Holomorphic curves and classification.
\newblock {\em arXiv preprint arXiv:2010.16330}, 2020.

\bibitem{MR4045357}
Samuel Lisi, Aleksandra Marinkovi\'{c}, and Klaus Niederkr\"{u}ger.
\newblock On properties of {B}ourgeois contact structures.
\newblock {\em Algebr. Geom. Topol.}, 19(7):3409--3451, 2019.

\bibitem{MR4278702}
Samuel Lisi and Chris Wendl.
\newblock Spine removal surgery and the geography of symplectic fillings.
\newblock {\em Michigan Math. J.}, 70(2):403--422, 2021.

\bibitem{zbMATH06182635}
Patrick Massot, Klaus Niederkr{\"u}ger, and Chris Wendl.
\newblock Weak and strong fillability of higher dimensional contact manifolds.
\newblock {\em Invent. Math.}, 192(2):287--373, 2013.

\bibitem{zbMATH04186536}
Dusa McDuff.
\newblock Symplectic manifolds with contact type boundaries.
\newblock {\em Invent. Math.}, 103(3):651--671, 1991.

\bibitem{moreno2017algebraic}
Agustin Moreno.
\newblock Algebraic torsion in higher-dimensional contact manifolds.
\newblock {\em arXiv preprint arXiv:1711.01562}, 2017.

\bibitem{zbMATH06862804}
Takahiro Oba.
\newblock Higher-dimensional contact manifolds with infinitely many {Stein}
  fillings.
\newblock {\em Trans. Am. Math. Soc.}, 370(7):5033--5050, 2018.

\bibitem{zbMATH02041154}
Burak Ozbagci and Andr{\'a}s~I. Stipsicz.
\newblock Contact 3-manifolds with infinitely many {Stein} fillings.
\newblock {\em Proc. Am. Math. Soc.}, 132(5):1549--1558, 2004.

\bibitem{zbMATH01465075}
Paul Seidel.
\newblock Lagrangian two-spheres can be symplectically knotted.
\newblock {\em J. Differ. Geom.}, 52(1):145--171, 1999.

\bibitem{zbMATH05688243}
Chris Wendl.
\newblock Strongly fillable contact manifolds and {{\(J\)}}-holomorphic
  foliations.
\newblock {\em Duke Math. J.}, 151(3):337--384, 2010.

\bibitem{zbMATH07442197}
Zhengyi Zhou.
\newblock {{\((\mathbb{R}\mathbb{P}^{2n-1},\xi_{\mathrm{std}})\)}} is not
  exactly fillable for {{\(n\neq 2^k\)}}.
\newblock {\em Geom. Topol.}, 25(6):3013--3052, 2021.

\bibitem{MR4566698}
Zhengyi Zhou.
\newblock On fillings of {$\partial (V\times {\Bbb {D}})$}.
\newblock {\em Math. Ann.}, 385(3-4):1493--1520, 2023.

\end{thebibliography}
\Addresses

\end{document}